\documentclass[11pt]{article}

\usepackage{amssymb,amsbsy,amsthm,amsmath,amsfonts}

\newcommand{\rr}{\stackrel {d}{=}}

\renewcommand{\Re}{{\rm I\kern-0.16em R}}

\def\@begintheorem#1#2{\trivlist \item[\hskip \labelsep{\bf #1\ #2}]}
\def\@opargbegintheorem#1#2#3{\trivlist
      \item[\hskip \labelsep{\bf #1\ #2\ (#3)}]}

\newtheorem{proposition}{Proposition}%[section] ger numering sektionsvis %denna med i latex
\newtheorem{definition}[proposition]{Definition}

\newtheorem{corollary}[proposition]{Corollary}

\newtheorem{remark}[proposition]{Remark}

\def\R{{\bf R}}
\def\R{{\bf R}}

\def\E{{\bf E}}

\begin{document}

\author{
Terhi Kaarakka
\\{\small Tampere University of Technology,}
\\{\small Mathematical Department,}
\\{\small FIN-33101 Tampere, Finland}
\\{\small email: terhi.kaarakka@tut.fi}
\and
Paavo Salminen\\{\small {\AA}bo Akademi University}
\\{\small Mathematical Department}
\\{\small FIN-20500 {\AA}bo, Finland} 
\\{\small email: phsalmin@abo.fi}
}
\vskip5cm

\title{On fractional Ornstein-Uhlenbeck processes
\vskip1cm}

\date{}

\maketitle

\begin{abstract}
In this paper we study Doob's transform of fractional Brownian
motion (FBM). It is well known that Doob's transform of
standard Brownian motion is identical in law with the
Ornstein-Uhlenbeck diffusion defined as the solution of the (stochastic) Langevin
equation where the driving process is a Brownian motion. It is also
known that Doob's transform of FBM and the process obtained
from the Langevin equation with FBM as the driving process are
different. However, also the first one of these can be described  as
a solution of a Langevin equation but now with some other driving
process than FBM. We are mainly interested in the properties of this new
driving process denoted $Y^{(1)}.$ We also study the solution of the Langevin equation with $Y^{(1)}$ 
as the driving process. Moreover, we show that the covariance of $Y^{(1)}$ 
grows linearly; hence, in this respect  $Y^{(1)}$ is more like a standard Brownian
motion than a FBM. In fact, it is proved that a properly scaled version of  $Y^{(1)}$ 
converges weakly to Brownian motion. 
\\ \\%\bigskip\noindent
{\rm Keywords: fractional Brownian motion, fractional Ornstein-Uhlenbeck process, long range dependence, 
short range dependence, covariance kernel, weak convergence}
\\ \\ %\bigskip\noindent
{\rm AMS Classification: 60G15, 60H05, 60G18}
\end{abstract}

\eject
\section{Introduction}
\label{sec0}
It is well known that the Ornstein-Uhlenbeck diffusion $U=\{U_t\,;\,
t\geq 0\}$ can be constructed as the unique strong solution of the Langevin SDE
\begin{equation}
\label{175} 
dU_t=-\alpha U_t\, dt+\,dB_t,
\end{equation}
where $\alpha>0$ and $B=\{B_t:t\geq 0\}$ is a standard Brownian motion initiated 
from 0. %, i.e., a FBM with $H=1/2.$ 
Solution of (\ref{175}) can be expressed as
\begin{equation}\label{e2.2}
U_t={\rm e}^{-\alpha t}\left(x+\int_0^t {\rm e}^{\alpha s}\,
dB_s\right),
\end{equation}
where $x$ is the (random) initial value of $U.$ Using partial integration,
the stochastic integral in (\ref{e2.2}) can be written as
\begin{equation}
\label{e2.3}
\int_0^s {\rm e}^{\alpha u} \,dB_u= {\rm e}^{\alpha s}B_s-\int_0^s
\alpha {\rm e}^{\alpha u} B_u\,du.
\end{equation}
The stationary distribution of $U$ is $N(0,1/2\alpha).$ 
Consequently, taking $x$ to be a normally distributed random variable with mean 0 and variance $1/(2\alpha)$ 
independent of $U$ gives us a stationary version of the
Ornstein-Uhlenbeck diffusion. 

Let $ B^{(-)} =\{ B^{(-)}_t:t\geq 0\}$ be another standard Brownian motion 
initiated from 0 and independent of $B.$ Introduce  for $t\in \R$ 
$$
\widehat B_t=
\begin{cases}B_t\,,\ & t\geq 0,\\
B^{(-)}_{-t}\,,\ & t\leq 0.\\
\end{cases}
$$
The process $\widehat B$ is sometimes called two-sided Brownian motion through 0. 
It is easily seen that 
$$
\xi:=\int_{-\infty}^0 {\rm e}^{\alpha s}\,d\widehat B_s
$$ 
is a normally distributed random variable with mean 0 and variance $1/(2\alpha).$ Since 
$
\lim_{s\to -\infty}\widehat B_s/s=0 
$
a.s., it follows via, e.g., (\ref{e2.3}) that $\xi$ is well defined. Choosing now 
$x=\xi$ allows us to write the stationary solution of (\ref{175}) in the form
\begin{equation*}
U_t={\rm e}^{-\alpha t} \int_{-\infty}^t {\rm e}^{\alpha s}\, d\widehat B_s.
\end{equation*}

There is also another well known construction of the
Ornstein-Uhlenbeck diffusion.   This is due to Doob \cite{doob42} and expresses the stationary Ornstein-Uhlenbeck diffusion 
$U$ (with time axis the whole $\R$) as a deterministic 
time change of a standard Brownian motion:
\begin{equation}
\label{161}
 U_t ={\rm e}^{-\alpha t}B_{a_t}, \quad t\in \R, %,\;H\in (0,1)
\end{equation}
where $\alpha>0$ and 
$
a_t:= {\rm e}^{\,2\alpha\, t}/2\alpha.
$ 
The covariance of $U$ is easily obtained from~(\ref{161})
\begin{equation}
\label{1611} 
\E\left( U_t\, U_s\right)= \frac
1{2\alpha}\, {\rm e}^{-\alpha(t-s)}, \quad  t\geq s.
\end{equation}

In this note we study fractional Ornstein-Uhlenbeck processes. These are processes constructed as $U$ above but 
now the Brownian motion is replaced with the fractional Brownian motion (FBM). It is known that the process obtained as the 
solution of the Langevin SDE with FBM as the driving process does not coincide with the process obtained as 
Doob's transform of FBM. In Cheridito et al. \cite{cheriditokawaguchimaejima03} it is proved 
that the covariance of the  former one behaves like the covariance of the increment process of FBM. In particular, 
if the Hurst parameter $H$ is bigger than $1/2$ the process is long range dependent. On the other hand, the covariance 
of Doob's transform\footnote{In \cite{cheriditokawaguchimaejima03} 
this transform is called Lamperti's transform (see Lamperti \cite{lamperti72}).} of FBM decays exponentially 
and, hence, the process is short range dependent for all values 
of $H\in(0,1).$ Our main contribution in this paper is to extract from Doob's transform the driving process, to study 
its properties and use the process
in the Langevin SDE to generate new kind of fractional Ornstein-Uhlenbeck processes. 

In the next section we discuss the basic properties of FBM important for our purposes. To make the paper more readable, 
we also recall some results from 
\cite{cheriditokawaguchimaejima03}. In the main section of the paper the new driving process 
is constructed and the solution of the associated Langevin SDE is introduced. The covariance 
of the driving process and also the covariance of the solution have kernel representations in case $H>1/2.$ 
It is proved then
that the driving process and the solution are short range dependent. Moreover, it is seen that it is possible to scale 
the driving process so that it converges weakly to a Brownian motion as the scaling parameter tends to infinity. 

\section{Preliminaries}
\label{sec1}

\subsection{Fractional Brownian motion}
\label{sec11}
Let $Z=\{ Z_t :t\geq 0\}$ %\in ( -\infty ,\infty ) \} $
be \emph{a fractional Brownian Motion,} FBM, with self-similarity
(or Hurst) parameter $H\in ( 0,1) $, that is, $Z$ is a centered
Gaussian
 process with the covariance function
\begin{equation}
\label{11}
\E(Z_t\,Z_s)=\frac{1}{2} \left(t^{2H}+s^{2H}-|t-s|^{2H} \right).
\end{equation}
Notice that 
$$ 
\E(Z^2_0)=0\ \text{ and }\ \E(Z^2_1)=1,
$$ 
and, hence, in
particular $Z_0=0$. Using Kolmogorov's continuity criterion it can
be proved that $Z$ has a continuous version; therefore, we take $Z$
to be continuous. In fact, $Z$ is locally H\"older continuous of
exponent $\alpha$ for all $\alpha<H$.

Fractional Brownian motion is $H$-self-similar in the sense
\begin{equation}
\label{12}
\{ Z_{\alpha t} :t\geq 0\}\ \rr\,\{ \alpha^H\,Z_{t} :t\geq 0\}
\quad{\rm for\  all\ } \alpha>0,
\end{equation}
where $\rr$ means that the right hand side and the left hand side are identical in law.
This follows from (\ref{11}) because the covariance function determines a mean zero Gaussian
distribution uniquely. Moreover, from (\ref{11}), for  ${t_2}>{t_1}>{s_2}>{s_1}$
\begin{eqnarray}
\label{136}
&&\hskip-1cm
\nonumber
\E\left((Z_{t_2}-Z_{t_1})(Z_{s_2}-Z_{s_1})\right)
\\
&&
=\frac 12\left((t_2-s_1)^{2H}-(t_1-s_1)^{2H}-(t_2-s_2)^{2H}+(t_1-s_2)^{2H}\right).
\end{eqnarray}
Since the function
$$
s\mapsto (t_2-s)^{2H}-(t_1-s)^{2H},\ s<t_1<t_2,
$$
is decreasing for $H>1/2,$ and increasing for $H<1/2$ it follows that the increments of $Z$ are
\begin{description}
\item{$\bullet$}\quad positively correlated if $H>1/2,$
\item{$\bullet$}\quad  negatively correlated if $H<1/2.$
\end{description}
Consider now the increment process of $Z$ defined as
$$
I_Z:=\{Z_{n+1}-Z_{n}\,:\,n=0,1,2,\dots \}.
$$
It is easily seen that $I_Z$ is a stationary second order stochastic
process
%Moreover, $I_Z$ is
%in stationary state since the distribution of $Z_{n+1}-Z_n$ does not
%depend on $n.$
and, from (\ref{136}),
\begin{equation}
\label{infbm}
\rho_{I_{Z}}(n):=\E\left(Z_1(Z_{n+1}-Z_n)\right)=H(2H-1)n^{-2(1-H)}+
O(n^{2H-3}).
\end{equation}
Next we recall the following definition (see Beran \cite{beran94} p. 6 and 42).
\begin{definition}\label{lrd}
Let $X=\{X_n: n=0,1,2,\dots\}$ be a stationary second order
stochastic process with mean zero and set
$
\rho_X(n):=\E\left(X_{i}X_{i+n} \right), 
$ 
where $i$ is arbitrary non-negative integer (by stationarity, $\rho_X(n)$ does not depend on~$i$). 
Then $X$ is called
\begin{description}
\item{$(i)$}\quad  \emph{long range dependent} if there exist $\alpha\in(0,1)$ and a constant $C>0$ such that
$\
\lim_{n\to\infty} \rho_X(n)/(C\,n^{-\alpha})=1,
$
%$\sum_{n=0}^\infty |\rho_X(n)|=+\infty,$
\item{$(ii)$}\quad  \emph{short range dependent} if 
$\
\lim_{k\to\infty}\sum_{n=0}^k \rho_X(n)
$ 
exists.
\end{description}
\end{definition}
\noindent
From Definition \ref{lrd} and formula (\ref{infbm}) it follows that the increment process $I_Z$ of the 
fractional Brownian  motion $Z$ is
\begin{description}
\item{$\bullet$}\quad  long range dependent if $H>1/2,$ 
%i.e. $$\sum_{n=0}^\infty |\rho_{I_{Z}}(n)|=\sum_{n=0}^\infty
%\rho_{I_{Z}}(n)=+\infty,$$
\item{$\bullet$}\quad  short range dependent if $H<1/2.$ 
%i.e. $\sum_{n=0}^\infty |\rho_{I_{Z}}(n)|<+\infty.$
\end{description}
Notice that, since $Z_0=0,$ we have
\begin{eqnarray}
\label{15} \lim_{N\to\infty}\E\left(Z_NZ_{1}\right)=
\sum_{n=0}^\infty \rho_{I_{Z}}(n).
\end{eqnarray}
%Consequently, in the long range dependent case, i.e. $H>1/2,$ it holds 
%$$
%\lim_{N\to\infty}\E\left(Z_NZ_{1}\right)=+\infty.
%$$

\subsection{Fractional Ornstein-Uhlenbeck processes of the first kind}
\label{sec12}

We replace now the Brownian motion $B$ in (\ref{175}) with the fractional Brownian motion $Z,$ 
and consider the SDE 
\begin{equation}
\label{121e} dU^{(Z,\alpha)}_t=-\alpha U^{(Z,\alpha)}_t\, dt+dZ_t.%{\rm e}^{-\alpha t}\,dZ_{b_t}.
\end{equation}
Analogously with (\ref{e2.2}), the solution can be expressed as
\begin{equation}
\label{e2.5}
U^{(Z,\alpha)}_t(x)={\rm e}^{-\alpha t}\left(x+\int_0^t {\rm e}^{\alpha
s}\, dZ_s\right)
\end{equation}
with some (random) initial value $x.$ The stochastic integral exists pathwise as a Riemann-Stiltjes integral 
(see Cheridito et al. \cite{cheriditokawaguchimaejima03}) and it holds
\begin{equation}\label{e2.6}
\int_0^s {\rm e}^{\alpha u} \,dZ_u= {\rm e}^{\alpha s}Z_s-\int_0^s
\alpha {\rm e}^{\alpha u} Z_u\,du.
\end{equation}
Furthermore, we introduce $\widehat Z,$ two-sided fractional Brownian motion through 0, 
and consider
\begin{equation}\label{e2.7}
\xi:=\int_{-\infty}^0
{\rm e}^{\alpha s}\, d\widehat Z_s.
\end{equation} 
Recall that the process given by
$$
Z^{(o)}_t:=
\begin{cases}0\,,\ & t=0,\\
t^{2H}Z_{1/t}\,,\ & t> 0.\\
\end{cases}
$$
is again a fractional Brownian motion. Therefore,
$$
\lim_{s\to -\infty} \widehat Z_s/|s|^{2H}=0\quad {\rm a.s.},%\frac{\widehat Z_s}{|s|^{2H}}=0
$$
and, consequently, $\xi$ is well-defined (via (\ref{e2.6})). Taking in (\ref{e2.5}) $x=\xi$ we write 
the solution in the form 
\begin{equation}
\label{e2.51}
U^{(Z,\alpha)}_t={\rm e}^{-\alpha t}\,\int_{-\infty}^t {\rm e}^{\alpha
s}\, d\widehat Z_s.
\end{equation}
Since the increments of $Z$ are stationary and the stochastic integral is a Riemann-Stiltjes integral it 
follows that the process $U^{(Z,\alpha)}$ is stationary. The stationary probability distribution, i.e., 
the distribution of $\xi,$ is normal with mean 0 and variance (see Cheridito et al. \cite{cheriditokawaguchimaejima03})
$$
\frac{\Gamma(2H+1)\sin(\pi H)}{\pi}\,\alpha^{-2H}\,\int_{0}^{+\infty}\frac{|x|^{1-2H}}{1+x^2}\,dx.
$$        
In case $H=1/2$, the variance equals $1/2\alpha,$ as it should. 
  
%We now wish to make the following definition.
\begin{definition}
\label{d2} The process $U^{(Z,\alpha)}$ given in (\ref{e2.51}) 
is called the \emph{stationary
fractional Ornstein-Uhlenbeck process of the first
kind}.%, for short, \emph{FOU$_1$}. %(in case $H>1/2$)
\end{definition}

Next we recall the asymptotic formula for the covariance of $U^{(Z,\alpha)}$
taken from \cite{cheriditokawaguchimaejima03} Theorem 2.3, which is then applied to derive 
the range dependence properties of $U^{(Z,\alpha)}.$

\begin{proposition}\label{pr2.1}
Let $H\in (0,\frac{1}{2}) \cup (\frac{1}{2},1]$ and $N=1,2,\ldots$.
Then for fixed $s\in \mathbb{R}$ and $t \to \infty$,
\begin{eqnarray}
\label{e2.8}
&&\nonumber
\hskip-1cm
\E(U^{(Z,\alpha)}_{s}U^{(Z,\alpha)}_{t+s})
\\
&&\hskip.8cm
=\frac{1}{2} \sum_{n=1}^{N} \alpha^{-2n}
\left(\prod_{k=0}^{2n-1}(2H-k)\right)t^{2H-2n}+ O(t^{2H-2N-2}).
\end{eqnarray}
\end{proposition}

\begin{proposition}
\label{pr2.2}
The stationary sequence $\{U^{(Z,\alpha)}_{n}:n=1,2,\dots\}$
(and, equivalently, the process $U^{(Z,\alpha)}$) is long range dependent
when $H>1/2$, and short range dependent when $H<1/2$.
\end{proposition}

\begin{proof}
Leading term of the sum in (\ref{e2.8}) is of the
order $t^{2H-2}$. Consequently,
\begin{eqnarray*}
\sum_{n=0}^{\infty}|\rho_{U^{(Z,\alpha)}}(n)|
&=&\sum_{n=0}^{\infty}|\E(U^{(Z,\alpha)}_{i}U^{(Z,\alpha)}_{i+n})|\\
&\simeq &\sum_{n=0}^{\infty} n^{2H-2},
\end{eqnarray*}
which, by Definition \ref{lrd}, gives the claim. 
\end{proof}

\section{Fractional Ornstein-Uhlenbeck processes of the second kind}
\label{sec3}
\subsection{Definition and some basic properties}
\label{sec31}
In this section we derive from Doob's transform of $Z$ a Gaussian process with stationary increments. This process is 
used as the driving process 
in the Langevin SDE. In this way we construct a new family of Gaussian processes which we call fractional Ornstein-Uhlenbeck 
processes of the second kind. This terminology can be justified by observing that in the standard Brownian case, 
i.e., $H=1/2,$  these processes coincide with the Ornstein-Uhlenbeck diffusions; as also do the fractional 
Ornstein-Uhlenbeck processes of the first kind introduced in Definition \ref{d2}.

Doob's transform of $Z$ is the process given by 
\begin{equation}
\label{16}
X^{(D,\alpha)}_t :={\rm e}^{-\alpha t}Z_{a_t}, \quad t\in \R, %,\;H\in (0,1)
\end{equation}
where $\alpha>0$ and 
$
a_t:=a(t,H):= H\,{\rm e}^{\,\alpha t/H}/\alpha.%\,\exp(\alpha t/H).
$
The covariance of $X$ can be computed from (\ref{11}). Indeed, for $t>s$ we have 
\begin{eqnarray}
\label{165} && \hskip-.5cm
\E(X^{(D,\alpha)}_tX^{(D,\alpha)}_s)\\
&&\hskip.5cm \nonumber = \frac{1}{2} \left(\frac
H\alpha\right)^{2H}\left( {\rm e}^{\alpha(t-s)}+{\rm
e}^{-\alpha(t-s)}-{\rm e}^{\alpha(t-s)}\left(1-{\rm
e}^{-\frac{\alpha(t-s)}{H}}\right)^{2H} \right).
\end{eqnarray}
Since $X^{(D,\alpha)}$ is a Gaussian process it follows herefrom that it is stationary. 
In particular, using the self-similarity property  of the fractional Brownian motion (see (\ref{12})) 
it is seen that  $X^{(D,\alpha)}_t$ is for
all $t$ normally distributed with mean 0 and variance
 $(H/\alpha)^{2H}.$ 
\begin{proposition}
\label{prop00}
The stationary process $\{X^{(D,\alpha)}_t:t\in\R\}$ is, for all $H\in(0,1),$ short range dependent.
\end{proposition}
\begin{proof}
Formula (\ref{165}) yields for a fixed $s$ as  $t\to\infty$
\begin{equation}
\label{166a} \E(X^{(D,\alpha)}_tX^{(D,\alpha)}_s)=O
%\left({\rm e}^{-\alpha r_Ht}\right)
\left(\exp\left(-\alpha\, \min\{1,(1-H)/H\}\,t\right)\right),
\end{equation}
and this implies the result.
\end{proof}
%%h
Consider now the process $Y^{(\alpha)}$ defined via
\begin{equation}
\label{1665} 
Y^{(\alpha)}_t:= \int_0^t {\rm e}^{-\alpha s}dZ_{a_s},
\end{equation}
where the integral is a (pathwise) Riemann-Stiltjes integral (cf. Section \ref{sec12}).
In case $H=1/2$, $Y^{(\alpha)}$ is, for all $\alpha,$ by L\'evy's theorem a standard Brownian motion.
Using $Y^{(\alpha)}$ the process $X^{(D,\alpha)}$ can be viewed as the solution of the equation 
\begin{equation}
\label{16652}
dX^{(D,\alpha)}_t=-\alpha X^{(D,\alpha)}_tdt+dY^{(\alpha)}_t.
\end{equation}
\begin{proposition}
\label{prop01}
For all  $\alpha>0$
\begin{equation}
\label{1667}
\{ \alpha^H Y^{(\alpha)}_{t/\alpha} :t\geq 0\}\ \rr\,\{ Y^{(1)}_{t} :t\geq 0\}.
\end{equation}
The process $Y^{(\alpha)}$ has stationary increments. 
\end{proposition}
\begin{proof}
Integrating by parts we obtain
\begin{equation}
\label{16651} 
Y^{(\alpha)}_t= \int_0^t {\rm e}^{-\alpha s}dZ_{a_s}= {\rm e}^{-\alpha t}\,Z_{a_t}-Z_{a_0}+ 
\alpha\int_0^t {\rm e}^{-\alpha s}\,Z_{a_s}\,ds
\end{equation}
Using (\ref{12})  -- the self-similary property of FBM -- the claimed identity in law (\ref{1667}) follows from (\ref{16651}). 
Moreover, the equality 
$$
\E\left(\left(Y^{(\alpha)}_{t_2}-Y^{(\alpha)}_{t_1}\right)\left(Y^{(\alpha)}_{s_2}-Y^{(\alpha)}_{s_1}\right)\right)
=
\E\left(\left(Y^{(\alpha)}_{t_2+h}-Y^{(\alpha)}_{t_1+h}\right)\left(Y^{(\alpha)}_{s_2+h}-Y^{(\alpha)}_{s_1+h}\right)\right)
$$
holds for $t_2>t_1>s_2>s_1>0$ and $h>0$ again by 
the self similarity of FBM 
and exploiting (\ref{16651}). Consequently, the increments of $Y^{(\alpha)}$ are stationary.
\end{proof}

Inspired by Proposition \ref{prop01}, we consider the Langevin
SDE with $Y^{(1)}$ as the driving process:
\begin{equation}
\label{1668} 
dU^{(D,\gamma)}_t= -\gamma U^{(D,\gamma)}_tdt + d Y^{(1)}_t,\quad \gamma> 0.
\end{equation}
The solution can be expressed (cf. (\ref{e2.51}))
\begin{eqnarray}
\label{1669}
&&\nonumber\hskip-.3cm
U^{(D,\gamma)}_t={\rm e}^{-\gamma t}\,\int_{-\infty}^t {\rm e}^{\gamma
s}\, d\widehat Y^{(1)}_s\\
&& \hskip.9cm
={\rm e}^{-\gamma t}\,\int_{-\infty}^t {\rm e}^{(\gamma-1)
s}dZ_{a_s},\quad \gamma>0,
\end{eqnarray}
where $\widehat Y^{(1)}$ stands for the two sided $Y^{(1)}$ process. To show that the stochastic integral term makes sense also 
for $\gamma\in (0,1]$ recall first that for all $\beta<H$ 
\begin{equation}
\label{1670}
\lim_{s\to 0} Z_s/|s|^\beta =0\quad {\rm a.s.}
\end{equation}
because $Z$ is H\"older continuous of order $\beta<H.$ Next for $T<0$ using partial integration 
%\begin{equation}
%\label{1670}
$$
\int_T^s {\rm e}^{(\gamma -1) u} \,dZ_{a_u}= {\rm e}^{(\gamma -1) s}Z_{a_s}-
{\rm e}^{(\gamma -1) T}Z_{a_T}-
(\gamma-1)\int_T^s
{\rm e}^{(\gamma-1) u} Z_{a_u}\,du,
$$
and by (\ref{1670}) the right hand side has a well defined limit as $T\to -\infty.$
%h

Since the increment process of $Y^{(1)}$ is 
stationary it follows that $U^{(D,\gamma)}$ is stationary and,
therefore, we have well justified the following
\begin{definition}
\label{d1} The process $U^{(D,\gamma)}$ defined in (\ref{1669}) or, equivalently,
via the SDE (\ref{1668}) is called {\rm the fractional Ornstein-Uhlenbeck
process of the second kind.}
\end{definition}

We conclude this section by characterizing the H\"older continuity of $Y^{(\alpha)}$ and $U^{(D,\gamma)}.$ The result 
holds for more general stochastic integrals with respect to $Z$ (see Z\"ahle \cite{zahle98}), but the following simple proof 
in our special case is perhaps worthwhile to present here.
\begin{proposition}
\label{prop0}
The sample paths of $Y^{(\alpha)}$ and $U^{(D,\gamma)}$ are (locally) H\"older continuous of order $\beta <H.$
\end{proposition}
\begin{proof}
From (\ref{16652}) we have
\begin{equation}
\label{122}
Y^{(\alpha)}_t=X^{(D,\alpha)}_t-X^{(D,\alpha)}_0+\int_0^t\alpha X^{(D,\alpha)}_s\,ds.
\end{equation}
Consequently, $t\mapsto Y^{(\alpha)}_t$ is continuous and the H\"older continuity properties of  $Y^{(\alpha)}$ and  $X^{(D,\alpha)}$ are the same.
Hence, let $T>0$ be given and consider for $s,t<T$ and $\beta>0$
\begin{eqnarray*}
&&\hskip-1cm
\frac{\left|X^{(D,\alpha)}_t-X^{(D,\alpha)}_s\right|}{\left|t-s\right|^\beta}=
\frac{\left|{\rm e}^{-\alpha t}Z_{a_t}-{\rm e}^{-\alpha
s}Z_{a_s}\right|}{\left|t-s\right|^\beta}\\
&&\hskip1.8cm
= \frac{\left|{\rm
e}^{-\alpha t}\left|Z_{a_t}-Z_{a_s}\right|+Z_{a_s}\left|{\rm
e}^{-\alpha t}-{\rm e}^{-\alpha
s}\right|\right|}{\left|t-s\right|^\beta}\\
&&\hskip1.8cm
\leq
\frac{\left|Z_{a_t}-Z_{a_s}\right|}{\left|a_t-a_s\right|^\beta}\frac{\left|a_t-a_s\right|^\beta}
{\left|t-s\right|^\beta} +|t-s|^{1-\beta}\, |Z_{a_s}|\frac{\left|{\rm
e}^{-\alpha t}-{\rm e}^{-\alpha
s}\right|}{\left|t-s\right|}\\
&&\hskip1.8cm
\leq K_T
\frac{\left|Z_{a_t}-Z_{a_s}\right|}{\left|a_t-a_s\right|^\beta}+C_T,
\end{eqnarray*}
where $K_T$ and $C_T$ are (random) constants which do not depend on $s$ and $t.$
The claim follows now from the fact that
the paths of FBM are  (locally) H\"older continuous of order  $\beta <H.$
Similarly, for the process 
$U^{(D,\gamma)}$ (starting from 0) we have 
$$
U^{(D,\gamma)}_t= Y^{(1)}_t-\gamma {\rm e}^{-\gamma t}\int_0^t {\rm e}^{\gamma s}\,Y^{(1)}_s\,ds,
$$
and  it follows that also $U^{(D,\gamma)}$ is H\"older continuous of order $\beta<H.$
\end{proof}

\subsection{Kernel representations of covariances and short range dependence} 
\label{ssec32}

We make now the following assumption valid throughout the rest of the paper
$$
1/2<H<1.
$$ 
In this case, as is easily checked, 
the covariance of the fractional Brownian motion has for ${t_2}>{t_1}$ and ${s_2}>{s_1}$ the kernel representation 
%\begin{eqnarray}
%\label{13} \E(Z_t\,Z_s)=\int_0^t
%\int_0^s\,H(2H-1)|u-v|^{2H-2}\,du\,dv
%\end{eqnarray}
%and, consequently,
%\begin{eqnarray}
$$
%\label{135}
\E\left((Z_{t_2}-Z_{t_1})(Z_{s_2}-Z_{s_1})\right)=\int_{t_1}^{t_2}
\int_{s_1}^{s_2}\,H(2H-1)|u-v|^{2H-2}\,du\,dv.
$$
%\end{eqnarray}
%Hence for $H>1/2$ increments are positively correlated. Notice also
%the formula
%\begin{eqnarray}
%\label{14} \E(dZ_t\,dZ_s):=\frac{\partial^2}{\partial t \partial
%s}\E(Z_t\,Z_s)\,dtds=H(2H-1)|t-s|^{2H-2}\,dt\,ds.
%\end{eqnarray}
In the next proposition we derive an analogous representation for the process 
$Y^{(1)}.$  The result is formulated for all values on $\alpha>0.$
\begin{proposition}
\label{p1}
The covariance of $Y^{(\alpha)}$ with $1/2<H<1$ has the kernel representation
\begin{eqnarray}
  \label{141} 
&&\hskip-1cm\nonumber
\E\left(\left(Y^{(\alpha)}_{t_2}-Y^{(\alpha)}_{t_1}\right)\,
\left(Y^{(\alpha)}_{s_2}-Y^{(\alpha)}_{s_1}\right)\right)\\
&&\hskip1.5cm 
=
    C(\alpha,H)\,\int_{t_1}^{t_2}\int_{s_1}^{s_2}\, \frac{{\rm
      e}^{-\alpha(1-H)(u-v)/H}}{|1 - {\rm e}^{-\alpha
      (u-v)/H}|^{2(1-H)}}\,du\,dv,
\end{eqnarray}
where ${t_2}>{t_1},$ ${s_2}>{s_1},$ and 
$$
C(\alpha,H):= H(2H-1)\left(\frac \alpha H\right)^{2(1-H)}.
$$
The kernel 
\begin{equation}
\label{1411}
r_{\alpha,H}(u,v):=C(\alpha,H)\,\frac{{\rm
      e}^{-\alpha(1-H)(u-v)/H}}{|1 - {\rm e}^{-\alpha
      (u-v)/H}|^{2(1-H)}}
\end{equation}
is symmetric, i.e., $r_{\alpha,H}(u,v)=r_{\alpha,H}(v,u)$ for all $u,v\in\R.$
\end{proposition}

\begin{proof}
Recall  the formula (see Gripenberg and Norros \cite{gripenbergnorros96} Proposition 2.2)
\begin{eqnarray}
\label{142}
&&\hskip-1cm\nonumber
\E\left(\int_\R f(s)dZ_s\int_\R g(t)dZ_t\right)
\\
&&\hskip2cm 
=H(2H-1)\int_\R\int_\R f(s)g(t)|s-t|^{2H-2}\,dtds, 
\end{eqnarray}
where $1/2<H<1$ and $f,g\in {\bf L}^2(\R)\cap{\bf L}^1(\R).$ Since
$$
Y^{(\alpha )}_t:= \int_0^t {\rm e}^{-\alpha s}dZ_{a_s}= {\rm e}^{-\alpha t}\,Z_{a_t}-Z_{a_0}+ 
\alpha\int_0^t {\rm e}^{-\alpha s}\,Z_{a_s}\,ds
$$
simple manipulations yield 
$$
Y^{(\alpha )}_t=H^H\,\int_{a_0}^{a_t}s^{-H}\,dZ_s.
$$
%with $a_t:=\frac H\alpha\,{\rm e}^{\alpha\, t/H}.$ 
Consequently, the claim follows by a straightforward 
application of (\ref{142}). 
\end{proof}
\begin{remark}
\label{abs}
Notice that the kernel $r_{\alpha,H}$ 
is in ${\bf L}^2([0,T]\times[0,T])$ if and only if $H>3/4.$ Consequently, for $Y^{(1)}$
we have similar absolute continuity properties as for fractional Brownian motion (see Cheridito 
\cite{cheridito03}). Namely, the measure induced by the process $\{B_t+Y^{(1)}_t:t\geq 0\},$ where $Y^{(1)}$ and 
the Brownian motion $B$ are assumed to be independent, is absolutely continuous with respect to the 
Wiener measure.  
\end{remark}
For the next result, recall from Proposition \ref{prop01} that the increments of $Y^{(\alpha)}$ 
are stationary.
%h
\begin{corollary}
\label{cor1}
The increments of $Y^{(\alpha)}$ are positively correlated. The increment process 
$I_Y:=\{Y^{(\alpha)}_{n+1}-Y^{(\alpha)}_{n}; n=0,1,\ldots\}$ is stationary and short range dependent.
\end{corollary}

\begin{proof}
From (\ref{141}) it follows immediately that the increments are positively correlated. Of course, we may also deduce from (\ref{141}) 
the stationarity of the increments of $Y^{(\alpha)}.$ To show that $I_Y$ 
is short range dependent consider
\begin{eqnarray}
\label{22} &&\hskip-.5cm \nonumber
\E\left(Y^o_{1}(Y^o_{n+1}-Y^o_{n})\right)=\int_n^{n+1}du\int_0^1dv\ r_{\alpha,H}(u,v)
\\
&&\hskip1cm
\nonumber
= \,
{\rm   e}^{-\alpha(1-H)n/H}\\
&&\hskip2cm
\nonumber
\times\,
\int_0^{1}du\int_0^1dv\, {\rm
  e}^{-\alpha(1-H)(u-v)/H}\,{|1 - {\rm e}^{-n/H}\,{\rm e}^{-\alpha
  (u-v)/H}|^{2(H-1)}}.
\end{eqnarray}
The integral term has a positive finite limit as $n\to\infty.$ Indeed, Lebesgue's dominated 
convergence theorem yields 
\begin{eqnarray*}
&&\hskip-1cm
\lim_ {n\to\infty}\int_0^{1}du\int_0^1dv\, {\rm
  e}^{-\alpha(1-H)(u-v)/H}\,{|1 - {\rm e}^{-n/H}\,{\rm e}^{-\alpha
  (u-v)/H}|^{2(H-1)}}\\
&&\hskip2cm
=
\int_0^{1}du\int_0^1dv\, {\rm
  e}^{-\alpha(1-H)(u-v)/H}.
\end{eqnarray*}
Consequently,
\begin{equation}
\label{215}
\rho_{Y^{(\alpha)}}(n):=\E\left(Y^{(\alpha)}_{1}(Y^{(\alpha)}_{n+1}-Y^{(\alpha)}_{n})\right)=O\left({\rm
  e}^{-(1-H) n/H}\right).
\end{equation}
and, hence,
\begin{eqnarray}
\label{23}
\lim_{N\to\infty}\E\left(Y^{(\alpha)}_NY^{(\alpha)}_{1}\right)=\sum_{n=0}^\infty \rho_{Y^{(\alpha)}}(n)<+\infty
\end{eqnarray}
completing the proof.
\end{proof}
%hh

Next we study the asymptotic behaviour of the variance and covariance of $Y^{(\alpha)}.$ 
For this, it is practical to rewrite the symmetric kernel $r_{\alpha,H}$ in (\ref{1411}) as 
$$
r_{\alpha,H}(t,s)=k_{\alpha,H}(t-s)
$$
with 
\begin{equation}
\label{2151}
k_{\alpha,H}(x):=C(\alpha,H)\,{\rm
  e}^{-\alpha(1-H)x/H}\,|1 - {\rm e}^{-\alpha
  x/H}|^{2H-2}.
\end{equation}
%\newpage
%tähän siirsin tavaraa
\bigskip

\begin{proposition}
\label{tpr21} The following formulas hold:
%\begin{eqnarray} 
%\label{t24} 
%&&\hskip-2.6cm  \E\left(
%(Y^{(\alpha)}_t)^2\right)= 2
%= 2H(2H-1)\left(\frac \alpha H\right)^{2(1-H)}
%\int_0^{t}(t-x)\, k_{\alpha,H}(x)\,dx,
%\end{eqnarray}
\begin{eqnarray} 
\label{t24b} 
&&\hskip-1.7cm  
\E\left(
(Y^{(\alpha)}_t-Y^{(\alpha)}_s)^2\right) =2
%= 2H(2H-1)\left(\frac \alpha H\right)^{2(1-H)}
\int_0^{t-s}(t-s-x)\, k_{\alpha,H}(x)\,dx,
\end{eqnarray}
\begin{eqnarray} 
\label{t24c} 
&&\hskip-1cm
\E\left( Y^{(\alpha)}_tY^{(\alpha)}_s\right) =  \int_0^{t}\, (t-x)\,
k_{\alpha,H}(x)\,dx  \\ \nonumber && \hskip1.5cm +\int_0^{s}\, (s-x)\,
k_{\alpha,H}(x)\,dx - \,\int_0^{t-s}\,(t-s-x) k_{\alpha,H}(x)\,dx.
\end{eqnarray}
Moreover,
\begin{equation}
\label{29a}
\E\left( (Y^{(\alpha)}_t)^2\right)=O(t) \text{ as } t \to
\infty,
\end{equation}
and
\begin{eqnarray}
\label{29} 
&&
%\nonumber
\hskip-1cm 
\lim_{t \to \infty}\E(Y^{(\alpha)}_tY^{(\alpha)}_s)
%\\
%&&\hskip1cm 
=  s
\int_0^{\infty}k_{\alpha,H}(x) dx  +  \int_0^s (s-x)\,k_{\alpha,H}(x) dx.
\end{eqnarray}

\end{proposition}

\begin{proof} We apply (\ref{141}) to obtain (\ref{t24b}): 
%\begin{eqnarray}
%\label{24} &&\hskip-1.4cm \nonumber \E\left( (Y^{(\alpha)}_t)^2\right)=
%\int_0^{t}du\int_0^tdv\,r_{\alpha,H}(u,v)
%\\
%&&\hskip.9cm \nonumber = 2\int_0^{t}du\int_0^udv\,r_{\alpha,H}(u,v)
%\\
%&&\hskip.9cm \nonumber = 2\,\,\int_0^{t}du\int_0^udv \
%k_{\alpha,H}(u-v)
%\\
%&&\hskip.9cm =2\,
%= 2H(2H-1)\left(\frac \alpha H\right)^{2(1-H)}
%\int_0^{t}(t-x)\, k_{\alpha,H}(x)\,dx.
%\end{eqnarray}

\begin{eqnarray}
\label{24} &&\hskip-2.2cm
\nonumber
\E\left( (Y^{(\alpha)}_t-Y^{(\alpha)}_s)^2\right)= \int_s^{t}du\int_s^tdv\,r_{\alpha,H}(u,v)
\\
&&\hskip1.2cm \nonumber = 2\int_s^{t}du\int_s^udv\,r_{\alpha,H}(u,v)
\\
&&\hskip1.2cm \nonumber = 2\,\,\int_s^{t}dy\int_0^{y-s}dx
\ k_{\alpha,H}(x)
\\
&&\hskip1.2cm \nonumber =
2\,\,\int_0^{t-s}dx\int_{x+s}^{t}dy \ k_{\alpha,H}(x)
\\
&&\nonumber\hskip1.2cm  =2\,
%= 2H(2H-1)\left(\frac \alpha H\right)^{2(1-H)}
\int_0^{t-s}(t-s-x)\, k_{\alpha,H}(x)\,dx.
\end{eqnarray}
Putting here $s=0$ and using  
$$
\int_0^\infty  k_{\alpha,H}(x)\,dx <\infty \quad{\rm and}\quad
\int_0^\infty x\, k_{\alpha,H}(x)\,dx <\infty
$$
yield (\ref{29a}). Furthermore, straightforward computations produces formula (\ref{t24c}) 
from  (\ref{t24b}). 
%Writing
%$$
%\begin{eqnarray}
%\label{t245} &&\hskip-1cm
%\nonumber
%\E\left( (Y^{(\alpha)}_t-Y^{(\alpha)}_s)^2\right)
%\\
%&&\hskip2cm
%= \E\left( (Y^{(\alpha)}_t )^2\right) +
%\E\left( (Y^{(\alpha)}_s)^2\right) - 2\E\left( Y^{(\alpha)}_tY^{(\alpha)}_s\right)
%\end{eqnarray}
%$$
%and using (\ref{t24}) yields (\ref{t24c}).
It remains to to prove (\ref{29}). Consider for $t>2s$
\begin{eqnarray}
\label{27} &&\hskip-2cm \nonumber \E\left(
(Y^{(\alpha)}_t-Y^{(\alpha)}_s)Y^{(\alpha)}_s\right)
= \int_s^{t}du\int_0^sdv\,r_{\alpha,H}(u,v)
\\
&&\hskip2cm \nonumber
= \,%2H(2H-1)\left(\frac \alpha H\right)^{2(1-H)}\,
\int_s^{t}du\int_0^sdv\ k_{\alpha,H}(u-v)
\\
&&\hskip2cm \nonumber = \int_0^{s}\, x\,
k_{\alpha,H}(x)\,dx + s\,\int_s^{t-s}\, k_{\alpha,H}(x)\,dx
\\
&&\nonumber\hskip3cm +\int_{t-s}^{t}(t-x)\, k_{\alpha,H}(x)\,dx.
\end{eqnarray}
Consequently,
\begin{eqnarray}
\label{28}\nonumber && \lim_{t \to \infty}\E\left((Y^{(\alpha)}_t-Y^{(\alpha)}_s)Y^{(\alpha)}_s\right)
\\
&& \hskip2cm \nonumber = \int_0^{s}\, x\,
k_{\alpha,H}(x)\,dx + s\,\int_s^{\infty}\, k_{\alpha,H}(x)\,dx
\\
&&\nonumber
 \hskip2cm  = s\,\int_0^{\infty}\,
k_{\alpha,H}(x)\,dx -\int_0^{s}\, (s-x)\, k_{\alpha,H}(x)\,dx,
\end{eqnarray}
from which (\ref{29}) easily follows. 
\end{proof}
\begin{remark}
\label{rem2}
The short range dependence property of $Y^{(\alpha)}$ also follows 
from (\ref{29}) since (recall that   $Y^{(\alpha)}_0=0$) 
$$
\sum_{n=0}^\infty \rho_{Y^{(\alpha)}}(n)=
\lim_{N\to\infty}\E\left(Y^{(\alpha)}_NY^{(\alpha)}_{1}\right)<+\infty.
$$
\end{remark}

\begin{proposition}
\label{p11}
The covariance of $U^{(D,\gamma)}$  has the kernel representation
\begin{eqnarray}
  \label{281}\nonumber 
&&
%\hskip-1.5cm
%\nonumber
\E\left(U^{(D,\gamma)}_tU^{(D,\gamma)}_s\right)
\\
&&\hskip1.5cm\nonumber 
=H^{2H-2}\,{\rm e}^{-\gamma(t+s)}\,\int_{-\infty}^{t}\int_{-\infty}^{s}\, 
\frac{{\rm e}^{(\gamma-1+\frac 1H)(u+v)}}{|{\rm e}^{ u/H}-{\rm e}^{v/H}|^{2(1-H)}}\,du\,dv.
\end{eqnarray}
\end{proposition}
\begin{proof}
As in the proof of Proposition \ref{p1}, we use also here formula (\ref {142}). However, now we need an extended version 
due to Pipiras and Taqqu \cite{pipirastaqqu00} stating that (\ref {142}) holds true for functions $f$ and $g$ satisfying 
\begin{equation}
\label{282}
\int_\R\int_\R |f(s)||g(t)||s-t|^{2H-2}\,dtds<\infty. 
\end{equation}
Consider 
\begin{equation}
\label{283}
U^{(D,\gamma)}_t={\rm e}^{-\gamma t}\,\int_{-\infty}^t {\rm e}^{(\gamma-1)
s}dZ_{a_s}
= H^{-(\gamma-1)H}\,
{\rm e}^{-\gamma t}\,\int_{0}^{a_t}s^{(\gamma-1)H}\,  dZ_{s}.
\end{equation}
To check that condition (\ref{282}) is valid for $f(s)=g(s)=s^{(\gamma-1)H}{\bf 1}_{(0,a_t)}(s)$
it is enough to show that 
\begin{eqnarray*}
&&
\int_0^1\int_0^1 (uv)^{(\gamma-1)H}|u-v|^{2H-2}\,dudv
\\
&&\hskip2cm
= 2\int_0^1du \,u^{(\gamma-1)H}\int_0^udv\, v^{(\gamma-1)H}(u-v)^{2H-2}<\infty. 
\end{eqnarray*}
The inner integral can be expressed in terms of the Beta-function
$$
\int_0^u v^{(\gamma-1)H}(u-v)^{2H-2}\, dv=u^{(\gamma+1)H-1} {\rm Beta}(1+(\gamma-1)H,2H-1).
$$
Consequently, 
$$
\hskip.5cm
\int_0^1\int_0^1 (uv)^{(\gamma-1)H}|u-v|^{2H-2}\,dudv = \frac 2{\gamma H}\, {\rm Beta}(1+(\gamma-1)H,2H-1),
$$
and condition (\ref{282}) holds. To verify the claimed kernel representation is now a straightforward computation 
using formula (\ref {142}).
\end{proof}

Recall from Corollary \ref{cor1} that the increment process of $Y^{(1)}$ is short range dependent, and that if $Y^{(1)}$ 
is used as the driving process in the Langevin equation the solution is the process $U^{(D,\gamma)}.$ 
In the next proposition we show that also $U^{(D,\gamma)}$ is short range dependent. Formula (\ref{284}) 
can be compared with the corresponding formula (\ref{166a}) for  $X^{(D,\alpha)}.$ In fact, (\ref{166a}) 
with $\alpha=1$ is (\ref{284}) with $\gamma=1,$ as it should.  
\begin{proposition}
\label{p12}
The rate of decay of the covariance of $U^{(D,\gamma)}$ is exponential. More precisely,
\begin{equation}
\label{284}
\E\left(U^{(D,\gamma)}_tU^{(D,\gamma)}_s\right)=
O\left(\exp\left(-\min\{\gamma,(1-H)/H\}\,t\right)\right),\quad {\rm as}\ t\to \infty.
\end{equation}
In particular, the stationary process  $U^{(D,\gamma)}$ is short range dependent. 
\end{proposition}
\begin{proof} 
Without loss of generality, we may take $s=0$ and, hence, consider
\begin{eqnarray*}
&&\hskip-1cm
\E\left(U^{(D,\gamma)}_tU^{(D,\gamma)}_0\right)
%\\
%&&
%\hskip3cm
=
H^{2H-2}\,{\rm e}^{-\gamma t}\,\int_{-\infty}^{t}\int_{-\infty}^{0}\, 
\frac{{\rm e}^{(\gamma-1+\frac 1H)(u+v)}}{|{\rm e}^{u/H}-{\rm e}^{v/H}|^{2(1-H)}}\,du\,dv
\\
&&
\hskip2.05cm 
= \Delta_1(t) +  \Delta_2(t),
\end{eqnarray*}
where, for some fixed $T>0,$
$$
\Delta_1(t):=H^{2H-2}\,{\rm e}^{-\gamma t}\,\int_{-\infty}^{T}\,du\int_{-\infty}^{0}\,dv\,
\frac{{\rm e}^{(\gamma-1+\frac 1H)(u+v)}}{|{\rm e}^{ u/H}-{\rm e}^{ v/H}|^{2(1-H)}}
$$
and
$$
\Delta_2(t):=H^{2H-2}\,{\rm e}^{-\gamma t}\,\int_{T}^{t}\, du\int_{-\infty}^{0}\,dv\,
\frac{{\rm e}^{(\gamma-1+\frac 1H)(u+v)}}{|{\rm e}^{ u/H}-{\rm e}^{ v/H}|^{2(1-H)}}.
$$
Clearly,
$$
\Delta_1(t)=O\left(\exp(-\gamma \,t)\right)\quad {\rm as}\ t\to +\infty.
$$ 
For the integral term in $\Delta_2(t)$ we have 
\begin{eqnarray*}
&&\hskip-1cm
\int_{T}^{t}\, du\int_{-\infty}^{0}\,dv\,
\frac{{\rm e}^{(\gamma-1+\frac 1H)(u+v)}}{|{\rm e}^{ u/H}-{\rm e}^{ v/H}|^{2(1-H)}}
%\\
%&&
%\hskip2cm
= 
\int_{T}^{t}\, du\int_{-\infty}^{0}\,dv\,
\frac{{\rm e}^{(\gamma+1-\frac 1H)\,u}\,{\rm e}^{(\gamma-1+\frac 1H)\,v}}
{\left(1-{\rm e}^{(v-u)/H}\right)^{2(1-H)}}.
\end{eqnarray*}
For $(u,v)\in(T,t)\times(-\infty,0)$ 
$$
1\leq \left(1-{\rm e}^{ (v-u)/H}\right)^{2(1-H)}\leq\left(1-{\rm e}^{- T/H}\right)^{2(1-H)},
$$
and, consequently, formula (\ref{284}) holds.
%$$
%\Delta_1(t)+\Delta_2(t)=O\left(\exp\left(-\min\{\gamma,\frac{1-H}H\}\,t\right)\right)\quad {\rm as}\ t\to +\infty.
%$$
\end{proof}

\subsection{Weak convergence of $Y^{(1)}$ to Brownian motion}
\label{ssec35}

In Proposition \ref{tpr21} it is proved that the growth of the variance of $Y^{(1)}_t$  is asymtotically linear 
as $t\to +\infty$ (see (\ref{29a})). This suggests that $Y^{(1)},$ when properly scaled, behaves asymtotically 
as a standard Brownian motion. We give the precise statement in the next proposition formulated 
for arbitrary $\alpha>0.$

\begin{proposition}
\label{351}
For $a>0$ define 
$$
{Z}^{(a,\alpha)}_t:=\frac{1}{\sqrt{a}}\,Y^{(\alpha)}_{at},\ t\geq 0,
$$ 
and let $B=\{B_t:t\geq 0\}$ denote standard Brownian motion started from 0. Then as $a\to +\infty$  
\[
\{{Z}^{(a,\alpha)}_t : t\geq 0\} \stackrel{\rm weakly}\Rightarrow \{\sigma B_t: t\geq 0\},
\]
where $\stackrel{\rm weakly}\Rightarrow$ stands for weak convergence in the space of continuous functions and  
$\sigma=\sigma(\alpha,H)$ is a non-random quantity depending only on $\alpha$ and $H$ (see (\ref{352})).  
\end{proposition}

\begin{proof}
We show first that the finite dimensional distributions of ${Z}^{(a,\alpha)}$ converge to the finite dimensional 
distributions of $\sigma B.$  Since ${Z}^{(a,\alpha)}$ is a Gaussian process with mean zero it is enough to verify 
the convergence of the covariance function. From (\ref{t24c}) in Proposition \ref{tpr21} we have for $t>s$
\begin{eqnarray*}
\E\left({Z}^{(a,\alpha)}_t {Z}^{(a,\alpha)}_s\right)
&=&\frac 1a\E\left(Y^{(\alpha)}_{at}Y^{(\alpha)}_{as}\right)\\
&=&\frac{1}{a} \,\Big( \int_0^{at}\, (at-x)\, k_{\alpha,H}(x)\,dx 
+\int_0^{as}\, (as-x)\,
k_{\alpha,H}(x)\,dx
\\ 
&& \hskip1.5cm 
- \,\int_0^{a(t-s)}\,(at-as-x) k_{\alpha,H}(x)\,dx \Big)
\end{eqnarray*}
with $k_{\alpha,H}$ defined in (\ref{2151}). Letting here $a\to +\infty$ yields, after some simple computations, 
\begin{eqnarray*}
\lim_{a \to\infty} \E\left({Z}^{(a)}_t{Z}^{(a)}_s\right)
&=&
%\lim_{a \to \infty}
%\frac{1}{a} \,\Big( \int_0^{at}\, (at-x)\, k_{\alpha,H}(x)\,dx  \\
%\nonumber && \hskip1.5cm +\int_0^{as}\, (as-x)\,
%k_{\alpha,H}(x)\,dx - \,\int_0^{a(t-s)}\,(at-as-x) k_{\alpha,H}(x)\,dx \Big)\\
% &=&\lim_{a \to \infty}
% \,\Big( \int_{a(t-s)}^{at}\,t\, k_{\alpha,H}(x)\,dx -\frac{1}{a}\int_{a(t-s)}^{at}\,x\, k_{\alpha,H}(x)\,dx \\
%\nonumber && \hskip1.5cm +\frac{1}{a}\int_{0}^{as}\,x\, k_{\alpha,H}(x)\,dx
%\int_0^{as}\, s\,
%k_{\alpha,H}(x)\,dx + \,\int_0^{a(t-s)}\,s\, k_{\alpha,H}(x)\,dx \Big)\\
2\,\,s\,\int_0^{\infty}\, k_{\alpha,H}(x)\,dx\\
&=&\kappa(\alpha,H)\,s,
\end{eqnarray*}
where 
$$
\kappa(\alpha,H):=2\,C(\alpha,H)\,\frac H\alpha\,{\rm Beta}(1-H,2H-1)
$$
%\\
%&=&2\,\,B(1-H,2H-1)\E(B_sB_t),
and ${\rm Beta}(1-H,2H-1)$ is the Beta function. Since $\E(B_tB_s)=s$ for $t>s$ we have proved the convergence 
of finite dimensional distributions of ${Z}^{(a,\alpha)}$ to the finite dimensional distributions of $\sigma B$
with 
\begin{equation}
\label{352}
\sigma=\sigma(\alpha,H)=\sqrt{\kappa(\alpha,H)}.
\end{equation}
To prove tightness, it is enough to verify (see, e.g., Lamperti \cite{lamperti62}) 
that there exists a constant $C$ (might depend on $\alpha$ and $H$) such that for all $a>0$ and $t>s$
$$
\Delta:=\E\left(\left({Z}^{(a,\alpha)}_t-{Z}^{(a,\alpha)}_s\right)^2\right)\leq C\,(t-s).
$$
We have by formula (\ref{t24b}) in Proposition \ref{tpr21}  
%h1
\begin{eqnarray*}
&&
\Delta= 
\frac 1 a\,\E\left(\left({Y}^{(\alpha)}_{at}-{Y}^{(\alpha)}_{as}\right)^2\right)
\\
&&
\hskip.5cm
=2\,\frac 1a\,\int_0^{a(t-s)}(a(t-s)-x)\, k_{\alpha,H}(x)\,dx
\\
&&
\hskip.5cm
\leq C\, (t-s)
\end{eqnarray*}
with, e.g., $C= 2\,\int_0^\infty \, k_{\alpha,H}(x)\,dx.$ This completes the proof.
\end{proof}
\vskip1cm
\noindent
{\bf Acknowledgement.} We thank Zhan Shi, Esko Valkeila and Marc Yor for discussions and comments on an early version of this paper.

\bibliographystyle{plain}
\bibliography{yor1}
%\bibliography{salm_vall}
\end{document}